\documentclass[11pt]{amsart}
\usepackage{amsmath,amssymb,bm, color}

\begin{document}
\newtheorem{theorem}{Theorem}[section]
\newtheorem{lemma}[theorem]{Lemma}
\newtheorem{corollary}[theorem]{Corollary}
\newtheorem{prop}[theorem]{Proposition}
\newtheorem{definition}[theorem]{Definition}
\newtheorem{remark}[theorem]{Remark}


 \def\ad#1{\begin{aligned}#1\end{aligned}}  \def\b#1{{\bf #1}} \def\hb#1{\hat{\bf #1}}
\def\a#1{\begin{align*}#1\end{align*}} \def\an#1{\begin{align}#1\end{align}}
\def\e#1{\begin{equation}#1\end{equation}} \def\t#1{\hbox{\rm{#1}}}
\def\dt#1{\left|\begin{matrix}#1\end{matrix}\right|}
\def\p#1{\begin{pmatrix}#1\end{pmatrix}} \def\c{\operatorname{curl}}
 \numberwithin{equation}{section} \def\P{\Pi_h^\nabla}

\def\bg#1{{\pmb #1}} 

\title  [Virtual elements]
   {Stabilizer-free polygonal and polyhedral virtual elements}

\author {Yanping Lin}
\address{Department of Applied Mathematics, The Hong Kong Polytechnic University, Hung Hom, Hong Kong}
\email{yanping.lin@polyu.edu.hk}
\thanks{Yanping Lin is supported in part by HKSAR GRF 15302922  and polyu-CAS joint Lab.}

\author {Mo Mu}
\address{Department of Mathematics, Hong Kong University of Science and Technology, Clear Water Bay, Kowloon, Hongkong, China}
\email{mamu@ust.hk}
\thanks{Mo Mu is supported in part by Hong Kong RGC CERG HKUST16301218. }

\author { Shangyou Zhang }
\address{Department of Mathematical
            Sciences, University
     of Delaware, Newark, DE 19716, USA. }
\email{szhang@udel.edu }

\date{}

\begin{abstract} Stabilizer-free $P_k$ virtual elements are constructed on 
  polygonal and polyhedral meshes.
Here the interpolating space is the space of continuous $P_k$ polynomials on a
   triangular-subdivision of each polygon, or a tetrahedral-subdivision of
   each polyhedron. 
With such an accurate and proper interpolation, the stabilizer of the virtual 
  elements is eliminated while the system is kept positive-definite.  
We show that the stabilizer-free virtual elements converge at the optimal order in
  2D and 3D.
Numerical examples are computed, validating the theory.

\end{abstract}



\maketitle \baselineskip=16pt

\section{Introduction}
The virtual element method is proposed and studied in  
\cite{Beirao,Beirao16, Cao-Chen, Cao-Chen2,  Chen1, Chen2, Chen3,
  Feng-Huang,Feng-Huang2, Huang1, Huang2, Huang3}. 
In this work, we construct stabilizer-free $P_k$ virtual elements on
  polygonal and polyhedral meshes.
 
We consider the Poisson equation, 
\an{ \label{p-e} \ad{ -\Delta  u & = f  && \t{in } \ \Omega, \\
            u&=0 && \t{on } \ \partial \Omega, } }
where $\Omega\subset \mathbb{R}^d$ is a bounded polygonal or polyhedral
     domain and $f\in L^2(\Omega)$.
The variation form reads:  Find $u\in H^1_0(\Omega)$ such that
\an{ \label{w-e} \ad{ (\nabla u, \nabla v) & = ( f,v)  && \forall v \in  H^1_0(\Omega),} }
where $(\cdot, \cdot)$ is the $L^2$ inner product on $\Omega$ 
    and we have $|v|_1^2=(\nabla v,\nabla v)$.

Let $\mathcal{T}_h=\{ K \}$ be a quasi-uniform polygonal or polyhedral mesh 
   on $\Omega$ with $h$ as the
  maximum size of the polygons or polyhedrons $K$. 
 Let $\mathcal{E}_h$ denote the set of edges $e$ in $\mathcal{T}_h$.
 In 3D, let $\mathcal{F}_h$ be the set of all face-polygons $F$ in $\mathcal{T}_h$.
For $k\ge 1$, the virtual element space is defined as
\an{ \label{t-V-h} \ad{ \tilde 
    V_h=\{ v\in H^1_0(\Omega) & : \tilde v \in \mathbb{B}_k(\mathcal{E}_h ),
   \Delta \tilde v|_K \in P_{k-2}(K) \}  
     \t{\ in 2D,  or as }\\
\tilde  V_h=\{ v\in H^1_0(\Omega) & : \tilde v \in \mathbb{B}_k(\mathcal{E}_h); 
    \Delta_F \tilde v|_F \in P_{k-2}(F), \ F\in \mathcal{F}_h; \\ &\qquad
   \Delta \tilde v|_K \in P_{k-2}(K) \}   } }
   in 3D, where $P_{-1}=\{0\}$,
   $ \mathbb{B}_k(\mathcal{E}_h)=\{ v\in C^0(\mathcal{E}_h) : v|_e\in P_k(e)
   \ \forall e\subset \mathcal{E}_h \}$, and $\Delta_F$ is the 2D Laplacian on
   the flat polygon $F$.
In computation, the  
   interpolated virtual finite element space on $\mathcal{T}_h$ is defined by
\an{ \label{V-h} V_h = \{ v_h=\Pi_h^\nabla \tilde v \ : \ v_h|_K \in \mathbb{V}_k(K), 
  \ K\in\mathcal{T}_h; \
   \tilde v\in \tilde V_h \}, }
where $\mathbb{V}_k(K)=P_k(K)$ for the standard virtual elements (and to be defined below
   in \eqref{V-k} for the new virtual element method),
   and $v_h=\Pi_h^\nabla \tilde v$ is the local $H^1$-projection:   
\a{ \left\{ \ad{ (\nabla(v_h-\tilde v), \nabla w_h)_K&=0\quad \forall 
     w_h=\Pi_h^\nabla \tilde w \in \mathbb{V}_k(K), \\
           \langle v_h-\tilde v, w_h\rangle_{\partial K}& =0 \quad
     \forall w_h\in P_k(K). } \right. }
The stabilizer-free virtual element equation reads:  
   Find $u_h=\P \tilde u\in V_h$ such that
\an{\label{f-e} (\nabla u_h,\nabla v_h)_h = (f,v_h) 
     \quad \forall \tilde v\in \tilde V_h, \ v_h=\P \tilde v, }
where $(\nabla u_h, \nabla v_h)_h=\sum_{K\in \mathcal{T}_h} (\nabla u_h, \nabla v_h)_K$.  
In 3D, to find the value of $\tilde v$ inside a face-polygon,  we 
  use the moments $\int_F \tilde v_h p_{k-2} dS$ instead 
     of the surface Laplacian values $\Delta_F \tilde v_h\in P_{k-2}(F)$, as the
   latter uniquely determines $\tilde v_h$ and consequently uniquely determines
   the $P_{k-2}$ moments of $\tilde v_h$ on $F$.

Because the dimension of $V_h$ is less than that of $\tilde V_h$ (equal only when $k=1$ on
   triangular and tetrahedral meshes),  
   the bilinear form in \eqref{f-e} is not positive-definite and 
   the equation does not have a unique solution.
Thus a discrete stabilizer must be added to the equation \eqref{f-e} if the
  interpolation space $\mathbb{V}_k(K)$ is defined to be $P_k(K)$.

With a stabilizer, many degrees of freedom do not fully contribute their approximation
   power as they are averaged into a smaller dimensional vector space.
To be a stabilizer free virtual element, the interpolation space must have at least no less 
   degrees of freedom on each element.
But raising the polynomial degree of $\mathbb{V}_k(K)$ in \eqref{V-h} does not work.
It works only for $P_1$ virtual elements in 2D with special treatment, cf.
  \cite{Berrone-0,Berrone-1}, where
   $\mathbb{V}_k(K)=P_{k+l(n)}(K)$ and $n$ is the number of edges of $K$,
   in the virtual element space \eqref{V-h}. 
Another stabilizer-free method for $k=1$ is proposed in \cite{Berrone} that
   $\mathbb{V}_k(K)=P_k(K)\cup H_l(K)$, where $H_l(K)$ is the set of 2D harmonic
   polynomials of degree $l$ or less, and $l$ depends on the number of
   edges of $K$.
This is an excellent idea because the $H_l$ harmonic polynomials may help
  to gather all boundary edge values while
  not destroying the gradient approximation, as harmonic polynomials 
    have vanishing Laplacian.
The same idea has been implemented in some other finite elements \cite{Al-Taweel,
 Sorokina1,Sorokina2}. 
But the method of \cite{Berrone} is also for 2D $P_1$ polygonal elements, as it
  is shown numerically not working for $k>3$ in \cite{Xu-Z}.

We propose to use macro-triangles or macro-tetrahedrons 
 $C^0$-$P_k$ spaces as the interpolation space $\mathbb{V}_k(K)$ in \eqref{V-h}.
This method was first used in \cite{Xu-Z} for $P_k$ triangular virtual elements only.
In \cite{Xu-Z}, each triangle $K$ is split into three triangles by connecting 
 its barycenter with the three vertices.   
 $K$ is called a Hsieh-Clough-Tocher macro-triangle \cite{Clough, Sorokina2,
 Xu-Zhang,  ZhangMG, Zhang3D}.
In this work, we extend the method to polygonal and polyhedral virtual elements.
It turns out that the triangular virtual element would be the most complicated case,
  as we have to introduce a new point in the subdivision in order to get
   a sufficiently large dimensional vector space.
For most polygons and polyhedrons we can subdivide them into triangles and
  tetrahedrons respectively without adding any new point, 
  when we have enough face-edge and 
   face-polygon degrees of freedom.
 
A different interpolation space $\mathbb{V}_k$ changes the
    quadrature rule for computing $(\nabla u_h, \nabla v_h)=
    (\nabla \Pi_h^\nabla \tilde u, \nabla \Pi_h^\nabla \tilde v)$.
Such an accurate local interpolation does not increase the computational cost,
   once the local stiffness matrix is generated.
On the other side, eliminating the stabilizer
   may reduce computational cost, and may  
   improve the condition number of the resulting linear system.
More importantly, a stabilizer-free method may utilize fully every degree of
  freedom in the discrete approximation.
Thus stabilizer-free methods may result in superconvergence.

The stabilizer is eliminated in the weak Galerkin finite element method  
  \cite{Al-Taweel-Wang,Feng-Zhang, Gao-Z,Mu1,Wang-Z20,
      Wang-Z21,Ye-Z20a,  Ye-Z20b,Ye-Z20c, Ye-Z21b, Ye-Z21c, Ye-Z21d}.
It is also eliminated 
    in the $H(\t{div})$ finite element method \cite{Mu21,Ye-Z21a, Ye-Z21f}.
The stabilizer-free  $C^0$ or $C^{-1}$
   nonconforming finite elements are constructed
      for the biharmonic equation \cite{Ye-Z20d,Ye-Z22a,Ye-Z22b}.
We have stabilizer-free  discontinuous Galerkin finite element methods
    \cite{Feng-Z21,Mu23,  Ye-Z20-1,Ye-Z20-2,Ye-Z20d}.
Without a stabilizer,
    two-order superconvergent weak Galerkin finite elements are found in \cite{Al-Taweel-Z21,
    Wang-Z23,Wang-Z23a,Ye-Z21g,Ye-Z23b, Ye-Z23e}.
Also two-order superconvergent stabilizer-free discontinuous Galerkin
    finite elements are constructed in \cite{Ye-Z22c,Ye-Z22d,Ye-Z23c} 
   for second order elliptic   equations.
One or two-order superconvergent weak Galerkin finite elements are 
   found for the Stokes equations in \cite{Mu21b,Ye-Z21e, Ye-Z22e}.
Four-order superconvergent weak Galerkin finite elements \cite{Ye-Z23d}
   and four-order  superconvergent discontinuous Galerkin
    finite elements \cite{Ye-Z22d,Ye-Z23e} are all stabilizer-free,
   for the biharmonic equation.
For example, a $P_3$ discontinuous finite element solution is as accurate as
  a $C^1$-$P_7$ finite element solution in solving a 2D biharmonic equation.

In this paper, we show
  that with the macro-triangle/tetrahedron
     interpolation, the stabilizer-free virtual element equation 
   \eqref{f-e} has a unique and quasi-optimal solution.
Numerical examples on the new stabilizer-free virtual elements are 
  computed, verifying the theory.

                \vskip .7cm
 
\section{The 2D interpolation} 

We define in this section the 2D macro-triangle interpolation space and show 
   that the stabilizer-free virtual element equation has a unique solution.

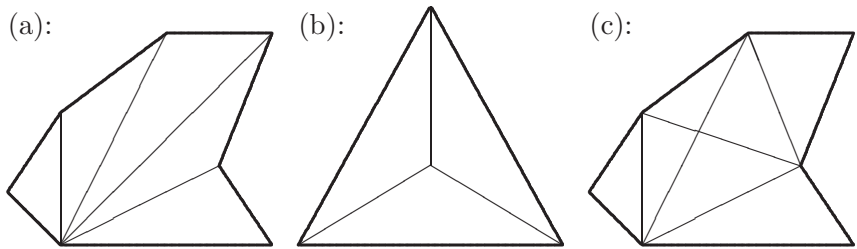
\begin{figure}\centering
\begin{picture}(330,100)(0,0)
\put(0,0){\begin{picture}(100,100)(0,0) 
      \put(20,0){\line(2,1){60}}\put(20,0){\line(1,1){80}}
      \put(20,0){\line(1,2){40}}\put(20,0){\line(0,1){50}}    \def\la{\circle*{0.3}}
   \put(0,80){(a):}

     \multiput(  20.00,   0.00)(   0.250,   0.000){320}{\la}
     \multiput( 100.00,   0.00)(  -0.139,   0.208){144}{\la}
     \multiput(  80.00,  30.00)(   0.093,   0.232){215}{\la}
     \multiput( 100.00,  80.00)(  -0.250,   0.000){160}{\la}
     \multiput(  60.00,  80.00)(  -0.200,  -0.150){200}{\la}
     \multiput(  20.00,  50.00)(  -0.139,  -0.208){144}{\la}
     \multiput(   0.00,  20.00)(   0.177,  -0.177){113}{\la}

    \end{picture} }
\put(110,0){\begin{picture}(100,100)(0,0) 
      \put(50,30){\line(-5,-3){50}}\put(50,30){\line(5,-3){50}}
      \put(50,30){\line(0,1){60}}    \def\la{\circle*{0.3}} 
     \multiput(   0.00,   0.00)(   0.250,   0.000){400}{\la}
     \multiput( 100.00,   0.00)(  -0.121,   0.219){411}{\la}
     \multiput(   0.00,   0.00)(   0.121,   0.219){411}{\la}
   \put(0,80){(b):}
    \end{picture} }
\put(220,0){\begin{picture}(100,100)(0,0) 
      \put(20,0){\line(2,1){60}}\put(80,30){\line(-2,5){20}}
      \put(20,0){\line(1,2){40}}\put(20,0){\line(0,1){50}} \put(80,30){\line(-3,1){60}}
   \def\la{\circle*{0.3}} 
     \multiput(  20.00,   0.00)(   0.250,   0.000){320}{\la}
     \multiput( 100.00,   0.00)(  -0.139,   0.208){144}{\la}
     \multiput(  80.00,  30.00)(   0.093,   0.232){215}{\la}
     \multiput( 100.00,  80.00)(  -0.250,   0.000){160}{\la}
     \multiput(  60.00,  80.00)(  -0.200,  -0.150){200}{\la}
     \multiput(  20.00,  50.00)(  -0.139,  -0.208){144}{\la}
     \multiput(   0.00,  20.00)(   0.177,  -0.177){113}{\la}
   \put(0,80){(c):}
    \end{picture} }
\end{picture} 
\caption{ (a) A polygon is subdivided without any new point. \quad
          (b) A triangle must be subdivided with one new point.\quad
          (c) A polygon is subdivided with one new point. }\label{f-3-p}
\end{figure}

Let $K$ be a 2D polygon.  The only requirement is that $K$ is subdivided into
  more than one tetrahedron.   If $K$ has only three sides, i.e., $K$ is a triangle,
   we add a barycenter point to the triangle, shown as in the Figure \ref{f-3-p}(b) 
   macro triangle on $K$.  
If $K$ is a polygon of four sides or more,  we usually can connect some vertices of
   $K$ to subdivide $K$ into a macro-triangle polygon, cf. Figure \ref{f-3-p}(a).
If needed, we can add one or two inner points to subdivide $K$, cf. 
    Figure \ref{f-3-p}(c),  where we intentionally add a new point for the
   purpose of illustration.
With the subdivision $K=\cup_{T_i\subset K} T_i$,  we define the
   interpolation space as, for $k\ge 1$, 
\an{\label{V-k} \mathbb{V}_k(K) =\{ v_h \in C^0(K) :
    v_h|_{T_i} \in P_k(T_i), \ T_i\subset K\}.  }
One can easily count the internal degrees of freedom of $\mathbb{V}_k(K)$ to get
\a{ \dim (\mathbb{V}_k\cap H^1_0(K))
     > \dim P_{k-2}(K).  }

The interpolation operator is defined to be the local $H^1$-projection, i.e.,
   $v_h=\P \tilde v \in \mathbb{V}_k $ such that $v_h|_{\partial K} = \tilde v$
   and 
\an{ \label{l-e} (\nabla (v_h-\tilde v), \nabla w_h)=0 \quad \forall w_h\in \mathbb{V}_k(K).
   }

\begin{lemma} \label{l-2d-w}
The interpolation operator
 $\P$ is well defined in \eqref{l-e} and it preserves $P_k$ polynomials,
\an{\label{p-p} \P \tilde v = \tilde v\quad \ \t{if } \ \tilde v\in P_k(K). }
\end{lemma}

\begin{proof} Because $\tilde v|_{\partial K}\in \mathbb{B}_k(\mathcal{E}_h)$, 
  $v_h$ can assume the boundary
  condition $v_h=\tilde v$ exactly on $\partial K$. The linear system of equations in \eqref{l-e}
 is a finite dimensional square system. The existence is implied by the uniqueness. 
 To show the uniqueness, we let $\tilde v=0$ in \eqref{l-e}.  
  Letting $w_h=v_h$ in \eqref{l-e}, we get
\a{ \nabla v_h =\b 0 \quad \t{ on } \ K. }
Thus $v_h=c$ is a constant on $K$.  As $v_h$ is continuous on edges,  $v_h=c$ is
a global constant on the whole domain.  By the boundary condition, we get $0=\tilde v|_
   {\partial \Omega}=v_h|_{\partial \Omega}=c$.
   Hence $v_h=0$ and \eqref{l-e} has a unique solution.  

If $\tilde v\in P_k(K)\subset \mathbb{V}_k(K)$, defined in \eqref{V-h}, then the solution   
   of \eqref{l-e} says, letting $w_h=v_h-\tilde v$, 
\a{ \nabla (v_h-\tilde v)=\b 0. }
Thus $v_h-\tilde v$ is a global constant which must be zero as it vanishes at all $\partial K$.
  \eqref{p-p} is proved. 
\end{proof}
                \vskip .7cm

\begin{lemma} \label{l-2d-e}
The stabilizer-free virtual element equation \eqref{f-e} has a unique solution,
where the interpolation $\P$ is defined in \eqref{l-e}.
\end{lemma}

\begin{proof} As both $\tilde u, \tilde v \in \tilde V_h$,
  \eqref{f-e} is a finite square  system of linear equations. The uniqueness of solution 
  implies the existence. 
 To show the uniqueness, we let $f=0$ and $\tilde v=\tilde u$ in \eqref{f-e}. 
   It follows that
\a{ |\P \tilde u|_{1,h} =0. }
Thus $\P \tilde u=c$ is  constant on each $K$.  But $\P \tilde u$ is continuous on the whole
 domain.  By the boundary condition, we get $0=\P \tilde u|_
   {\partial \Omega}=c$.  That is,
\an{\label{p-u-0}  \P \tilde u=0 \ \t{ and } \ 
                      \tilde u|_{\partial K}=\P \tilde u=0.  }

For $k=1$, $\tilde u$ has no internal degree of freedom, and the lemma is proved 
    by \eqref{p-u-0},
\a{ \tilde u=0, \ \t{ if } \ k=1.  }

For $k\ge 2$, let 
\an{\label{b-K} b_K &=\sum_{i\in \mathcal{N}_2} \phi_i\in  H^1_0(K)\cap \mathbb{V}_k(K), }
where $\mathcal{N}_2$ is the set of all internal mid-edge points of $\{ T_i\}$, $K=\cup T_i$,
   and $\phi_i$ is the $P_2$ Lagrange nodal basis at node $i$.
Then $b_K>0$ inside polygon $K$ if it does not have any added internal point, cf. Figure
  \ref{f-3-p}.  Otherwise, $b_K>0$ inside $K$ except at one or two
    internal points where $b_K=0$.
 
On one polygon $K$, by \eqref{l-e}, \eqref{p-u-0}  and 
   integration by parts, we have
\an{\label{l-e1} (-\Delta \tilde u, w_h) =
     (\nabla \tilde u, \nabla w_h)=0 \quad \forall w_h\in H^1_0(K)\cap \mathbb{V}_k(K).  }
 By the space $\tilde V_h$ definition \eqref{t-V-h}, we denote
\an{\label{p-k-2} p_{k-2} = -\Delta \tilde u \in P_{k-2}(K).  }
 Let the $w_h$ in \eqref{l-e1} be
\an{\label{w-h} w_h= p_{k-2} b_K
   \in H^1_0(K)\cap \mathbb{V}_k(K),  }
  where the positive $P_2$ bubble $b_K$ is defined in \eqref{b-K}. 
With the $w_k$ in \eqref{w-h}, we get from \eqref{l-e1} and \eqref{p-k-2} that
\a{ \int_K p_{k-2}^2 b_K d \b x =0.  }
As $b_K>0$ inside $K$ (other than 1 or 2 possibly internal points),  it follows that
\a{ p_{k-2}^2 =0 \ \t{ and } \ p_{k-2}  =0 \ \t{ on } \ K. }
By \eqref{p-u-0} and \eqref{p-k-2}, $\Delta \tilde  u=0$ in $K$ and $\tilde u=0$ on 
   $\partial K$.  Thus, by the unique solution of the Laplace equation,  $\tilde u=0$.
  The lemma is proved. 
\end{proof}
                \vskip .7cm

\section{The 3D interpolation} 

We define in this section the 3D macro-tetrahedron interpolation space and show 
   that the stabilizer-free virtual element equation has a unique solution when using
   the interpolation.

\def\cbt{\def\la{\circle*{0.3}} 
      \multiput(  40.00, 100.00)(  -0.093,  -0.232){430}{\la}
     \multiput(  40.00, 100.00)(   0.129,  -0.214){466}{\la}
     \multiput(  40.00, 100.00)(   0.177,  -0.177){282}{\la}
     \multiput(   0.00,   0.00)(   0.250,   0.000){400}{\la}
     \multiput(  90.00,  50.00)(   0.049,  -0.245){203}{\la}
     \multiput(  88.25,  49.03)(  -2.622,  -1.457){ 34}{\la}} 

\def\cbtf{\def\la{\circle*{0.3}}
     \multiput(  46.60,  33.30)(  -0.025,   0.249){268}{\la}
     \multiput(  46.60,  33.30)(   0.212,  -0.132){251}{\la}
     \multiput(  46.60,  33.30)(  -0.203,  -0.145){229}{\la}
     \multiput(  76.60,  50.00)(   0.250,   0.000){ 53}{\la}
     \multiput(  76.60,  50.00)(   0.106,  -0.226){220}{\la}
     \multiput(  76.60,  50.00)(  -0.148,   0.202){247}{\la}
     \multiput(  61.37,  16.09)(  -2.902,  -0.761){ 20}{\la}
     \multiput(  65.12,  15.78)(   2.733,  -1.236){ 12}{\la}
     \multiput(  64.55,  18.16)(   1.873,   2.343){ 14}{\la}
     \multiput(  45.30,  50.00)(   3.000,   0.000){ 14}{\la}
     \multiput(  41.99,  48.49)(  -1.964,  -2.268){ 22}{\la}
     \multiput(  43.17,  52.00)(  -0.198,   2.993){ 16}{\la}} 
\def\cbtii{\def\la{\circle*{0.3}}
      \multiput(  55.63,  36.78)(  -2.799,  -1.079){  2}{\la}
     \multiput(  59.17,  38.60)(   2.510,   1.643){  6}{\la}
     \multiput(  56.00,  38.82)(  -2.252,   1.982){  6}{\la}
     \multiput(  58.03,  35.57)(   0.802,  -2.891){  6}{\la}}
\def\cbtiii{\def\la{\circle*{0.3}}
     \multiput(  47.91,  48.49)(   1.972,  -2.261){  4}{\la}
     \multiput(  74.93,  48.90)(  -2.510,  -1.643){  6}{\la}
     \multiput(  62.77,  18.53)(  -0.802,   2.891){  6}{\la}
     \multiput(  44.80,  48.68)(   2.252,  -1.982){  6}{\la} }
\def\cbti{\def\la{\circle*{0.3}} 
      \multiput(  55.82,  36.41)(  -2.513,  -1.639){ 22}{\la}
     \multiput(  56.87,  39.40)(  -0.949,   2.846){ 18}{\la}
     \multiput(  59.37,  38.22)(   2.800,   1.077){ 10}{\la}
     \multiput(  59.00,  36.18)(   2.250,  -1.985){ 18}{\la}} 
\begin{figure}\centering
\begin{picture}(320,100)(0,0)
\put(0,0){\begin{picture}(100,100)(0,0)   \def\la{\circle*{0.3}}
   \put(0,80){(a):}\cbt
    \end{picture} }
\put(110,0){\begin{picture}(100,100)(0,0)  \cbt\cbtf
   \put(0,80){(b):}    
    \end{picture} }
\put(220,0){\begin{picture}(100,100)(0,0)  \cbt\cbti\cbtii\cbtf
   \put(0,80){(c):}
    \end{picture} }
\end{picture} 
\caption{ (a) A tetrahedron $K$. \quad
          (b) By adding a barycenter point to each face-triangle,  
               the face triangles of $K$ are subdivided into twelve triangles.  \quad
          (c) By adding one barycenter point of $K$,
         the tetrahedron $K$ is subdivided into twelve tetrahedrons. }\label{f-2-t}
\end{figure}
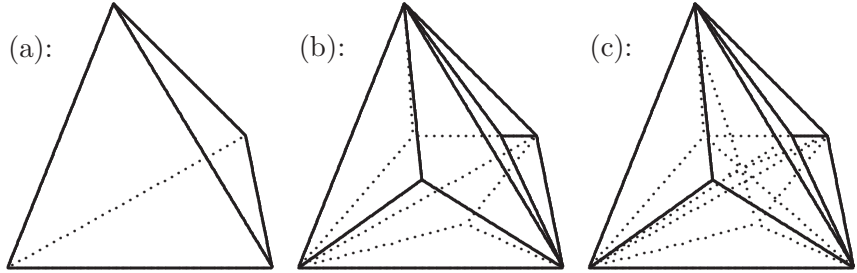

Let $K$ be a 3D polyhedron.  The first requirement is that each face-polygon $F$ must
  be subdivided into more than one triangle.
The subdivision of polygons is defined in the last section.
For example, if a face polygon $F$ has only three sides, i.e., $F$ is a triangle,
   we must add a barycenter point to subdivide it into three triangles,
   cf.  Figure \ref{f-2-t}(b).  
However if a face polygon has more than three edges,  we usually can subdivide
   it into triangles easily, cf. Figure \ref{f-3-c}(b).
After each face polygon is subdivided into more than one triangle,  the next
   requirement in subdividing $K$ is that every resulting tetrahedron has
  at least two face-triangles inside $K$. 

For example, after cutting each face-triangle of a tetrahedron $K$ into
   three triangles,  we add one more internal point to cut $K$ into twelve 
   tetrahedrons, cf. Figure \ref{f-2-t}(c).

\def\cbe{\def\la{\circle*{0.3}} 
     \multiput(   0.00,   0.00)(   0.250,   0.000){240}{\la}
     \multiput(   0.00,   0.00)(   0.000,   0.250){240}{\la}
     \multiput(  60.00,  60.00)(  -0.250,   0.000){240}{\la}
     \multiput(  40.00,  90.00)(  -0.200,  -0.150){200}{\la}
     \multiput(  40.00,  90.00)(   0.250,   0.000){240}{\la}
     \multiput( 100.00,  30.00)(   0.000,   0.250){240}{\la}
     \multiput( 100.00,  30.00)(  -0.200,  -0.150){200}{\la}
     \multiput(  60.00,  60.00)(   0.000,  -0.250){240}{\la}
     \multiput(  60.00,  60.00)(   0.200,   0.150){200}{\la}
     \multiput(  42.00,  30.00)(   3.000,   0.000){ 20}{\la}
     \multiput(  38.40,  28.80)(  -2.400,  -1.800){ 16}{\la}
     \multiput(  40.00,  32.00)(   0.000,   3.000){ 20}{\la} }
\def\dbc{ \def\la{\circle*{0.3}} 
     \multiput(  60.00,  60.00)(  -0.139,   0.208){144}{\la}
     \multiput(  60.00,  60.00)(  -0.177,  -0.177){339}{\la}
     \multiput(  60.00,  60.00)(   0.200,  -0.150){200}{\la}
     \multiput(  41.11,  28.34)(   1.664,  -2.496){ 12}{\la}
     \multiput(  38.40,  31.20)(  -2.400,   1.800){ 16}{\la}
     \multiput(  41.11,  31.66)(   1.664,   2.496){ 12}{\la}
     \multiput(  41.41,  31.41)(   2.121,   2.121){ 28}{\la} }
\def\dbd{ \def\la{\circle*{0.3}} 
     \multiput(  48.51,  43.66)(  -2.230,  -2.007){ 22}{\la}
     \multiput(  48.08,  45.57)(  -2.873,   0.862){ 16}{\la}
     \multiput(  49.57,  46.95)(  -0.651,   2.929){ 14}{\la}
     \multiput(  51.49,  46.34)(   2.230,   2.007){ 22}{\la}
     \multiput(  51.92,  44.43)(   2.873,  -0.862){ 16}{\la}
     \multiput(  50.43,  43.05)(   0.651,  -2.929){ 14}{\la} }

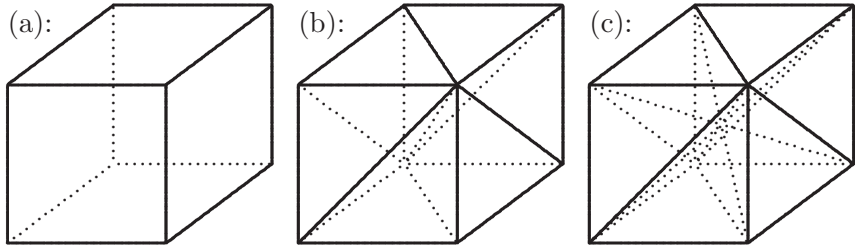
\begin{figure}\centering
\begin{picture}(330,100)(0,0)
\put(0,0){\begin{picture}(100,100)(0,0)   \def\la{\circle*{0.3}}
   \put(0,80){(a):}\cbe
    \end{picture} }
\put(110,0){\begin{picture}(100,100)(0,0)  \cbe \dbc
   \put(0,80){(b):}    
    \end{picture} }
\put(220,0){\begin{picture}(100,100)(0,0) \cbe\dbc \dbd
   \put(0,80){(c):}
    \end{picture} }
\end{picture} 
\caption{ (a) A polyhedron $K$. \quad
          (b) A polyhedron is subdivided into six tetrahedrons 
                 without adding any point.\quad
          (c) A polyhedron is subdivided into twelve tetrahedrons with one new
             point and without adding any face point. }\label{f-3-c}
\end{figure}

For example, for the polyhedron $K$ of a cube in Figure \ref{f-3-c}(a),  
    we cut each face-polygon
   into two triangles without adding any point, and we cut $K$ into six tetrahedrons without
   adding any internal point, cf.  Figure \ref{f-3-c}(b).

For example, for the polyhedron $K$ of a cube in Figure \ref{f-3-c}(a), 
   we can also subdivide it 
  by cutting each face-polygon
   into two triangles without adding any point, and  cutting $K$ into twelve
   tetrahedrons with one internal point, cf.  Figure \ref{f-3-c}(c).

\begin{figure}\centering
\begin{picture}(230,100)(0,0)
\put(0,0){\begin{picture}(100,100)(0,0)   \def\la{\circle*{0.3}}
   \put(0,80){(a):}\cbe
    \end{picture} }
\put(130,0){\begin{picture}(100,100)(0,0)  \cbe \dbc  \dbd
   \put(0,80){(b):}     \multiput(   0.00,  60.00)(   0.239,   0.072){417}{\la}
     \multiput(   0.00,  60.00)(   0.177,  -0.177){339}{\la}
     \multiput( 100.00,  90.00)(  -0.102,  -0.228){393}{\la}
     \multiput(  41.41,  88.59)(   2.121,  -2.121){ 28}{\la}
     \multiput(  39.19,  88.17)(  -1.218,  -2.741){ 32}{\la}
     \multiput(   1.92,   0.57)(   2.873,   0.862){ 34}{\la}
     \multiput(  50.00,  17.00)(   0.000,   3.000){ 22}{\la}
     \multiput(  22.00,  45.00)(   3.000,   0.000){ 20}{\la}
     \multiput(  31.60,  31.20)(   2.400,   1.800){ 16}{\la}
    \end{picture} }
\end{picture} 
\caption{ (a) A polyhedron $K$. \quad
          (b) A polyhedron is subdivided into twenty four tetrahedrons 
        with one new
             point each face-polygon and one internal point. }\label{f-2-c}
\end{figure}
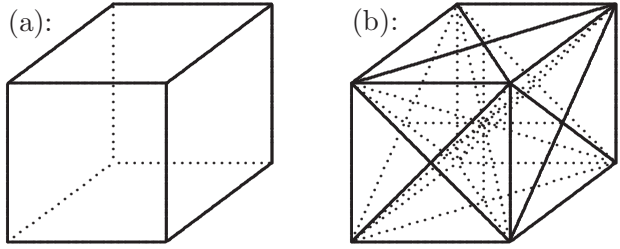

For example, for the polyhedron $K$ of a cube in Figure \ref{f-2-c}(a), 
   we can also subdivide it 
  by cutting each face-polygon
   into four triangles with one added point on each face-polygon, and  
   cutting $K$ into twenty four
   tetrahedrons with one additional  point inside $K$, cf.  Figure \ref{f-2-c}(b).

In the analysis,  we assume the same face-polygon subdivision on the two polyhedrons
   sharing the polygon.
In the computation, the subdivisions of a shared polygon on the two sides can be
  different as the interpolation and the computation on the two polyhedrons are
  independent of each other.
We can extend the theory easily to cover the case that different triangulations on 
   a face-polygon of two polyhedrons,  as both interpolations are the
   2D $H^1$ projection of same $P_{k-2}$ moments.
 
With a proper tetrahedral subdivision of $K=\cup_{T_i\subset K} T_i$, 
   cf. Figures \ref{f-2-t}--\ref{f-2-c},  we define the
   interpolation space on $K$ as, for $k\ge 1$, in \eqref{V-k}, again in 3D.

The interpolation operator is defined by two steps.
On each face polygon $F\in\mathcal{F}_h$,  we solve an $H^1$ projection problem that
    $v_h|_F=\P \tilde v\in \mathbb{V}_k(F)$ (the restriction of $\mathbb{V}_k(K)$ on
     $F$) satisfying
\an{\label{v-F} \ad{ (\nabla_F(v_h|_F -\tilde v), \nabla_F w_h) &=0 \quad \forall w_h\in
             \mathbb{V}_k(F)\cap H^1_0(F), \\
                   v_h|_F -\tilde v&=0 \quad \t{ on } \partial F, } }
where $\nabla_F$ is the 2D face gradient.
This way, the boundary value of $\P \tilde v$ is determined on $\partial K$.
The interpolation in 3D is defined as the 3D local $H^1$-projection, i.e.,
   $v_h=\P \tilde v \in \mathbb{V}_k $ such that 
\an{\label{l-e-3} \ad{ (\nabla(v_h  -\tilde v), \nabla w_h) &=0 \quad \forall w_h\in
             \mathbb{V}_k(K)\cap H^1_0(K), \\
                  v_h- v_h|_F&=0 \quad \t{ on all }  F\in \partial K, } }
where $v_h|_F$ is defined in \eqref{v-F}.

\begin{lemma} 
The interpolation operator
 $\P$ is well defined in \eqref{l-e-3} and it preserves $P_k$ polynomials,
\an{\label{p-p3} \P \tilde v = \tilde v\quad \ \t{if } \ \tilde v\in P_k(K). }
\end{lemma}

\begin{proof} Because $\tilde v|_{\partial F}\in \mathbb{B}_k(\mathcal{E}_h)$, 
  $v_h$ can assume the boundary
  condition $v_h=\tilde v$ exactly on $\partial F$, where $F$ is a face-polygon
   in the polyhedral mesh. 
  As we have proved in Lemma \ref{l-2d-w}, $v_h|_F$ is well-defined in
    \eqref{v-F}.  Further by Lemma \ref{l-2d-w}, 
 \an{\label{f-p-k}  v_h|_F = p_k|_F \quad\t{ if } p_k=\tilde v\in P_k(K).  }

The linear system of equations in \eqref{l-e-3}
 is a finite dimensional square system, after the boundary condition is enforced. 
   The existence is implied by the uniqueness. 
 To show the uniqueness, we let $\tilde v=0$ in \eqref{l-e-3}.  
 By Lemma \ref{l-2d-w}, $v_h|_F = 0$.   Letting $w_h=v_h$ in \eqref{l-e-3}, we get
\a{ \nabla v_h =\b 0 \quad \t{ on } \ K. }
Thus $v_h=c$ is one constant on all tetrahedrons of $K$. 
 As $v_h$ is continuous,  by the zero boundary condition,  $v_h=0$ and \eqref{l-e-3} 
    has a unique solution.  

If $\tilde v=p_k\in P_k(K)\subset \mathbb{V}_k(K)$,  then  
   \eqref{l-e-3} says, letting $w_h=v_h-p_k$, 
\a{ \nabla (v_h-p_k)=\b 0. }
Thus $v_h-p_k$ is a constant on $K$, which must be zero by \eqref{f-p-k}.
  \eqref{p-p3} is proved. 
\end{proof}
                \vskip .7cm

\begin{lemma} 
The stabilizer-free virtual element equation \eqref{f-e} has a unique solution,
where the interpolation $\P$ is defined in \eqref{l-e-3}.
\end{lemma}

\begin{proof} As both $\tilde u, \tilde v \in \tilde V_h$,
  \eqref{f-e} is a finite square  system of linear equations. 
  We only need to show the uniqueness, by letting
       $f=0$ and $\tilde v=\tilde u$ in \eqref{f-e}. 
  As in the 2D proof of Lemma \ref{l-2d-e}, we have
\a{ |\P \tilde u|_{1,h} =0 \quad \t{ and } \ \P \tilde u=0. } 

For $k=1$, $\tilde u$ has no internal degree of freedom on each face polygon $F$, and 
  $\tilde v|_F=\P \tilde u|_F=0$. 
Further $\tilde u$ has no internal degree of freedom on each polyhedron $K$, and 
  $\tilde v|_K=\P \tilde u|_K=0$.  The lemma is proved.

For $k\ge 2$, by the proof of Lemma \ref{l-2d-e}, as each face polygon is subdivided into
   more than one tetrahedron,  we have $\tilde v|_F=\P \tilde u|_F=0$ on every face polygon
   $F$.
Next, as every tetrahedron has at least two internal face triangles,  we define an internal
   $P_2$ bubble by 
\an{\label{b-K3} b_K &=\sum_{i\in \mathcal{N}_2} \phi_i\in  H^1_0(K)\cap \mathbb{V}_k(K), }
where $\mathcal{N}_2$ is the set of all internal mid-edge points of $\{ T_i\}$, $K=\cup T_i$,
   and $\phi_i$ is the $P_2$ Lagrange nodal basis at node $i$.
As every tetrahedron has such an internal $P_2$ node (which is the shared-edge mid-point
   of two internal face triangles),
     $b_K>0$ inside polyhedron $K$ if it does not have any added internal point. 
Otherwise, $b_K>0$ inside $K$ except at one or two
    internal points of $K$ where $b_K=0$.
 
On one polyhedron $K$, let 
\an{\label{w-h3} w_h= p_{k-2} b_K
   \in H^1_0(K)\cap \mathbb{V}_k(K),  }
  where the positive $P_2$ bubble $b_K$ is defined in \eqref{b-K3},
   and $p_{k-2}=-\Delta \tilde u\in P_{k-2}(K)$. 
With the integration by parts, we get from \eqref{l-e1} and \eqref{w-h3} that
\a{ \int_K p_{k-2}^2 b_K d \b x =-\int_K \nabla \tilde u \nabla w_h d \b x =0.  }
As $b_K>0$ inside $K$ (other than 1 or 2 possibly internal points),  it follows that
\a{ p_{k-2}^2 =0 \ \t{ and } \ p_{k-2}  =0 \ \t{ on } \ K. }
As $\Delta \tilde  u=0$ in $K$ and $\tilde u=0$ on 
   $\partial K$, by the unique solution of the Laplace equation,  $\tilde u=0$.
  The lemma is proved. 
\end{proof}
                \vskip .7cm

\section{Convergence}
           
We show that the stabilizer-free virtual element solution converges
  at the optimal order, in this
  section.

\begin{theorem}  Let the solution of \eqref{w-e} be $ u\in H^{k+1}\cap H^1_0(\Omega)$.
 Let the stabilizer-free virtual element solution of \eqref{f-e} be $u_h$.   Then the
   discrete solution converges at the optimal order with the following error estimate, 
  \an{ \label{h-1}  |  u- u_h |_{1}    \le Ch^{k} | u|_{k+1}.  } 
\end{theorem}
                        
\begin{proof} As $w_h\in V_h \subset H^1_0(\Omega) $,  subtracting \eqref{f-e} from
   \eqref{w-e}, we obtain
\a{  (\nabla (u-u_h), \nabla  w_h) =0\quad \forall w_h\in   V_{h }. } 
By the Schwarz inequality,  we get that
\a{    |   u- u_h|_{1}^2  
     & =  (\nabla(  u-  u_h), \nabla(  u- I_h u))\\ 
     &\le |   u- u_h|_{1} |   u- I_h u |_{1} \le Ch^{k} |u|_{k+1}
 |   u- u_h|_{1} ,} 
      where $ I_h  u$ is the Scott-Zhang interpolation on
  subdivided triangular mesh or tetrahedral mesh, cf. \cite{Scott-Zhang}. 
The theorem is proved.
\end{proof}

To get the optimal order $L^2$ error estimate,  we assume a full regularity for the dual
   problem that the solution of 
\an{ \label{dual2} \ad{ -\Delta w &= u-u_h \quad \t{ in } \ \Omega, \\
            w &=0  \quad \t{ on } \ \partial \Omega, } }
satisfies
\an{\label{r} |w|_2 \le C \|u-u_h\|_0.  }

\begin{theorem}  Let the solution of \eqref{w-e} be $ u\in H^{k+1}\cap H^1_0(\Omega)$. 
 Let the stabilizer-free virtual element solution of \eqref{f-e} be $u_h$.   Then the
   discrete solution converges at the optimal order with the following $L^2$ error estimate, 
     assuming  \eqref{r},  
  \a{   \|  u- u_h \|_{0}    \le Ch^{k+1} | u|_{k+1}.  } 
\end{theorem}
                        
\begin{proof} Let $w_h=\P \tilde w$ be the virtual element solution of \eqref{dual2}.
By \eqref{dual2}, \eqref{r} and \eqref{h-1},  we get
\a{ \|u-u_h\|_0^2 &=(\nabla w, \nabla (u-u_h) ) = 
(\nabla (w-w_h), \nabla (u-u_h) ) \\
  & \le C h |w|_2  h^{k} |u|_{k+1} \le C h^{k+1} |u|_{k+1}\|u-u_h\|_0. }
Canceling a $\|u-u_h\|_0$ on both sides, we proved the optimal-order $L^2$ error bound. 
\end{proof}

                \vskip .7cm

\section{Numerical test}

We solve numerically the Poisson equation \eqref{p-e}
   on domain $\Omega=(0,1)\times(0,1)$, where an exact solution is chosen as
\an{\label{s-1} u(x,y)=\sin (\pi x) \sin (\pi y).  }

We test the $P_k$ ($k=1,2,3,4,5$) stabilizer-free 
    virtual elements on pentagonal meshes shown in Figure \ref{mesh2}.

 \def\csp{\begin{picture}(  100.,  100.)(  0.,  0.) 
     \def\la{\circle*{0.3}}     \multiput(   0.00,   0.00)(   0.250,   0.000){200}{\la}
     \multiput(  50.00,   0.00)(   0.000,   0.250){200}{\la}
     \multiput(  25.00,  25.00)(   0.177,   0.177){141}{\la}
     \multiput(   0.00,  50.00)(   0.177,  -0.177){141}{\la}
     \multiput(   0.00,   0.00)(   0.000,   0.250){200}{\la}
     \multiput( 100.00,   0.00)(   0.000,   0.250){200}{\la}
     \multiput(  75.00,  75.00)(   0.177,  -0.177){141}{\la}
     \multiput(  50.00,  50.00)(   0.177,   0.177){141}{\la}
     \multiput(  50.00,   0.00)(   0.250,   0.000){200}{\la}
     \multiput(  50.00, 100.00)(   0.250,   0.000){200}{\la}
     \multiput(  50.00,  50.00)(   0.000,   0.250){200}{\la}
     \multiput( 100.00,  50.00)(   0.000,   0.250){200}{\la}
     \multiput(   0.00,  50.00)(   0.000,   0.250){200}{\la}
     \multiput(   0.00, 100.00)(   0.250,   0.000){200}{\la} \end{picture}} 
\begin{figure}[ht]
 \begin{center} \setlength\unitlength{1.0pt}
\begin{picture}(220,113)(0,0) 

  { \setlength\unitlength{1pt}
 \multiput(0,0)(50,0){1}{\multiput(0,0)(0,50){1}{\csp}}}
 
   \put(0,103){Grid 1:} \put(120,103){Grid 2:} 
  { \setlength\unitlength{0.5pt}
  \multiput(240,0)(100,0){2}{\multiput(0,0)(0,100){2}{\csp}} }

 \end{picture}\end{center}
\caption{The first two levels of pentagonal meshes for the computation in
     Table \ref{ta1}.}
\label{mesh2}
\end{figure}
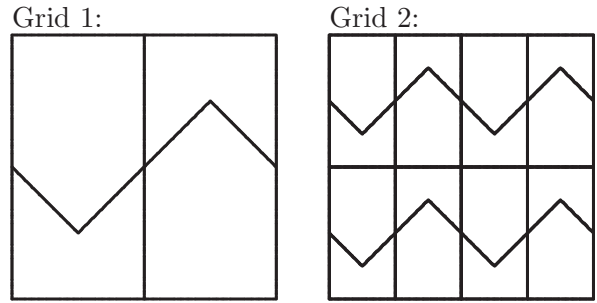

In Table \ref{ta1},  we compute the $P_1$--$P_5$ stabilizer-free 
    virtual elements solutions for
  \eqref{s-1} on the pentagonal meshes shown in Figure \ref{mesh2}.
All virtual element solutions  converge at rates of the optimal order in both 
  $L^2$ and $H^1$ norms.

\begin{table}[ht]
  \centering  \renewcommand{\arraystretch}{1.1}
\caption{The error profile for \eqref{s-1} on meshes shown in Figure \ref{mesh2}.}
  \label{ta1} 
\begin{tabular}{c|cc|cc}
\hline 
Grid &   $\|\Pi^\nabla_h u-u_h\|_{0}$  &  $O(h^r)$ 
   &   $|\Pi^\nabla_h u-u_h|_{1}$ & $O(h^r)$ 
  \\ \hline 
    &  \multicolumn{4}{c}{ By the $P_1$ stabilizer-free virtual element. } \\
\hline   
 7&   0.4462E-04 & 2.00&   0.5834E-02 & 1.00 \\
 8&   0.1116E-04 & 2.00&   0.2916E-02 & 1.00 \\
 9&   0.2789E-05 & 2.00&   0.1458E-02 & 1.00 \\
\hline
    &  \multicolumn{4}{c}{ By the $P_2$ stabilizer-free virtual element. } \\
\hline   
 7&   0.1930E-06 & 3.00&   0.1131E-03 & 2.00 \\
 8&   0.2413E-07 & 3.00&   0.2826E-04 & 2.00 \\
 9&   0.3016E-08 & 3.00&   0.7066E-05 & 2.00 \\
\hline
    &  \multicolumn{4}{c}{ By the $P_3$ stabilizer-free virtual element. } \\
\hline   
 6&   0.2486E-07 & 4.00&   0.8973E-05 & 3.00 \\
 7&   0.1554E-08 & 4.00&   0.1122E-05 & 3.00 \\
 8&   0.9716E-10 & 4.00&   0.1402E-06 & 3.00 \\
\hline
    &  \multicolumn{4}{c}{ By the $P_4$ stabilizer-free virtual element. } \\
\hline   
 5&   0.1051E-07 & 5.00&   0.1977E-05 & 4.00 \\
 6&   0.3286E-09 & 5.00&   0.1236E-06 & 4.00 \\
 7&   0.1027E-10 & 5.00&   0.7724E-08 & 4.00 \\
\hline
    &  \multicolumn{4}{c}{ By the $P_5$ stabilizer-free virtual element. } \\
\hline   
 3&   0.7591E-06 & 6.02&   0.4741E-04 & 4.98 \\
 4&   0.1181E-07 & 6.01&   0.1488E-05 & 4.99 \\
 5&   0.1846E-09 & 6.00&   0.4659E-07 & 5.00 \\
\hline 
    \end{tabular}%
\end{table}%

 \def\csq{\begin{picture}(  100.,  100.)(  0.,  0.) 
     \def\la{\circle*{0.3}}     \multiput(   0.00,   0.00)(   0.250,   0.000){200}{\la}
     \multiput(  44.44,  14.29)(   0.091,  -0.233){ 61}{\la}
     \multiput(  22.22,  22.22)(   0.235,  -0.084){ 94}{\la}
     \multiput(  18.18,  44.44)(   0.045,  -0.246){ 90}{\la}
     \multiput(   0.00,  50.00)(   0.239,  -0.073){ 76}{\la}
     \multiput(   0.00,   0.00)(   0.000,   0.250){200}{\la}
     \multiput( 100.00,   0.00)(   0.000,   0.250){200}{\la}
     \multiput(  81.82,  50.00)(   0.250,   0.000){ 72}{\la}
     \multiput(  72.73,  28.57)(   0.098,   0.230){ 93}{\la}
     \multiput(  44.44,  14.29)(   0.223,   0.113){126}{\la}
     \multiput(  50.00,   0.00)(   0.250,   0.000){200}{\la}
     \multiput(  50.00, 100.00)(   0.250,   0.000){200}{\la}
     \multiput(  44.44,  85.71)(   0.091,   0.233){ 61}{\la}
     \multiput(  44.44,  85.71)(   0.207,  -0.140){136}{\la}
     \multiput(  72.73,  66.67)(   0.120,  -0.219){ 75}{\la}
     \multiput( 100.00,  50.00)(   0.000,   0.250){200}{\la}
     \multiput(   0.00,  50.00)(   0.000,   0.250){200}{\la}
     \multiput(  18.18,  44.44)(   0.045,   0.246){ 90}{\la}
     \multiput(  22.22,  66.67)(   0.190,   0.163){117}{\la}
     \multiput(   0.00, 100.00)(   0.250,   0.000){200}{\la}
     \multiput(  44.44,  85.71)(   0.074,  -0.239){149}{\la}
     \multiput(  44.44,  14.29)(   0.074,   0.239){149}{\la}

 \end{picture}}

\begin{figure}[ht]
 \begin{center} \setlength\unitlength{1.0pt}
\begin{picture}(220,113)(0,0) 

  { \setlength\unitlength{1pt}
 \multiput(0,0)(50,0){1}{\multiput(0,0)(0,50){1}{\csq}}}
 
   \put(0,103){Grid 1:} \put(120,103){Grid 2:} 
  { \setlength\unitlength{0.5pt}
  \multiput(240,0)(100,0){2}{\multiput(0,0)(0,100){2}{\csq}} }

 \end{picture}\end{center}
\caption{The first two levels of hexagonal grids for the computation in
     Table \ref{ta2}.}
\label{mesh1}
\end{figure}
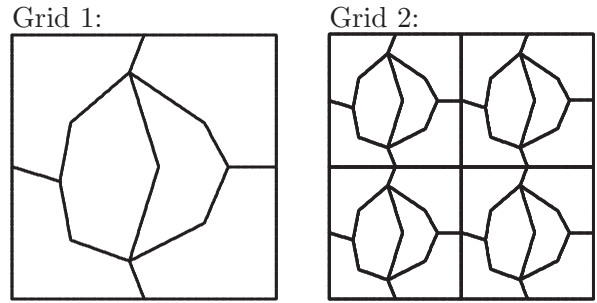

In Table \ref{ta2}, we compute the $P_1$--$P_5$ stabilizer-free 
    virtual elements solutions for
  \eqref{s-1} on the hexagonal meshes shown in Figure \ref{mesh1}.
All virtual element solutions  converge at rates of the optimal order in both 
  $L^2$ and $H^1$ norms.

\begin{table}[ht]
  \centering  \renewcommand{\arraystretch}{1.1}
\caption{The error profile for \eqref{s-1} on meshes shown in Figure \ref{mesh1}.}
  \label{ta2} 
\begin{tabular}{c|cc|cc}
\hline 
Grid &   $\|\Pi^\nabla_h u-u_h\|_{0}$  &  $O(h^r)$ 
   &   $|\Pi^\nabla_h u-u_h|_{1}$ & $O(h^r)$ 
  \\ \hline 
    &  \multicolumn{4}{c}{ By the $P_1$ stabilizer-free virtual element. } \\
\hline   
 6&   0.1142E-03 & 2.00&   0.1255E-01 & 1.00 \\
 7&   0.2855E-04 & 2.00&   0.6273E-02 & 1.00 \\
 8&   0.7137E-05 & 2.00&   0.3136E-02 & 1.00 \\
\hline
    &  \multicolumn{4}{c}{ By the $P_2$ stabilizer-free virtual element. } \\
\hline   
 6&   0.1011E-05 & 3.00&   0.4338E-03 & 2.00 \\
 7&   0.1265E-06 & 3.00&   0.1085E-03 & 2.00 \\
 8&   0.1581E-07 & 3.00&   0.2712E-04 & 2.00 \\
\hline
    &  \multicolumn{4}{c}{ By the $P_3$ stabilizer-free virtual element. } \\
\hline   
 6&   0.2132E-07 & 4.00&   0.1081E-04 & 3.00 \\
 7&   0.1332E-08 & 4.00&   0.1351E-05 & 3.00 \\
 8&   0.8329E-10 & 4.00&   0.1689E-06 & 3.00 \\
\hline
    &  \multicolumn{4}{c}{ By the $P_4$ stabilizer-free virtual element. } \\
\hline   
 5&   0.1200E-07 & 4.99&   0.3029E-05 & 4.00 \\
 6&   0.3756E-09 & 5.00&   0.1894E-06 & 4.00 \\
 7&   0.1175E-10 & 5.00&   0.1185E-07 & 4.00 \\
\hline
    &  \multicolumn{4}{c}{ By the $P_5$ stabilizer-free virtual element. } \\
\hline   
 3&   0.1177E-05 & 5.98&   0.8909E-04 & 4.97 \\
 4&   0.1842E-07 & 6.00&   0.2800E-05 & 4.99 \\
 5&   0.2895E-09 & 5.99&   0.8776E-07 & 5.00 \\
\hline 
    \end{tabular}%
\end{table}%

We solve the 3D Poisson equation \eqref{p-e}
   on domain $\Omega=(0,1)^3$, where an exact solution is chosen as
\an{\label{s-2} u(x,y,z)= 2^6 (x-x^2)(y-y^2)(z-z^2).  }

\begin{figure}[ht]
\begin{center}
 \setlength\unitlength{1pt}
    \begin{picture}(320,122)(0,3)
    \put(0,0){\begin{picture}(110,110)(0,0)\put(0,102){Grid 1:}
       \multiput(0,0)(80,0){2}{\line(0,1){80}}  \multiput(0,0)(0,80){2}{\line(1,0){80}}
       \multiput(0,80)(80,0){2}{\line(1,1){20}} \multiput(0,80)(20,20){2}{\line(1,0){80}}
       \multiput(80,0)(0,80){2}{\line(1,1){20}}  \multiput(80,0)(20,20){2}{\line(0,1){80}}
      \end{picture}}
    \put(110,0){\begin{picture}(110,110)(0,0)\put(0,102){Grid 2:}
       \multiput(0,0)(40,0){3}{\line(0,1){80}}  \multiput(0,0)(0,40){3}{\line(1,0){80}}
       \multiput(0,80)(40,0){3}{\line(1,1){20}} \multiput(0,80)(10,10){3}{\line(1,0){80}}
       \multiput(80,0)(0,40){3}{\line(1,1){20}}  \multiput(80,0)(10,10){3}{\line(0,1){80}}
      \end{picture}}
    \put(220,0){\begin{picture}(110,110)(0,0)\put(0,102){Grid 3:}
       \multiput(0,0)(20,0){5}{\line(0,1){80}}  \multiput(0,0)(0,20){5}{\line(1,0){80}}
       \multiput(0,80)(20,0){5}{\line(1,1){20}} \multiput(0,80)(5,5){5}{\line(1,0){80}}
       \multiput(80,0)(0,20){5}{\line(1,1){20}}  \multiput(80,0)(5,5){5}{\line(0,1){80}}
      \end{picture}}

    \end{picture}
    \end{center}
\caption{ The first three grids for the computation in Table \ref{ta3}.  }
\label{grid3d}
\end{figure}
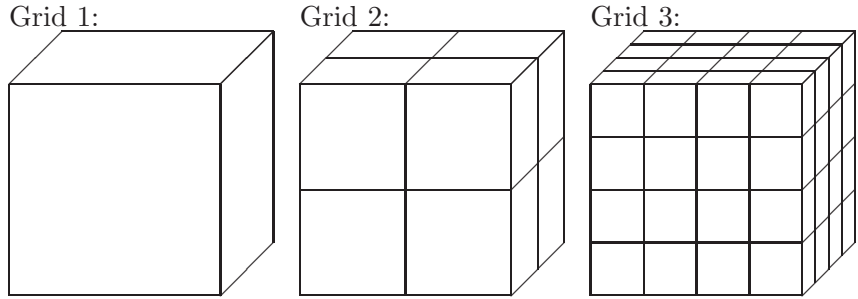

In Table \ref{ta3}, we compute the 3D $P_1$--$P_5$ stabilizer-free 
    virtual elements solutions for
  \eqref{s-2} on the cubic meshes shown in Figure \ref{grid3d}.
All virtual element solutions  converge at rates of the optimal order in both 
  $L^2$ and $H^1$ norms. 
In particular,  we have one order superconvergence in $H^1$ semi-norm for the
   $P_1$ stabilizer-free  virtual element solutions.
Also, we have one order superconvergence in both $H^1$ semi-norm and $L^2$ for the
   $P_2$ stabilizer-free  virtual element solutions.
But we do not have a theory for these superconvergences. 
It is surprising that the $P_2$ solutions are more accurate than the $P_3$ solutions
   in Table \ref{ta3}.

\begin{table}[ht]
  \centering  \renewcommand{\arraystretch}{1.1}
\caption{The error profile for \eqref{s-2} on cubic meshes shown in Figure \ref{grid3d}.}
  \label{ta3} 
\begin{tabular}{c|cc|cc}
\hline 
Grid &   $\|\Pi^\nabla_h u-u_h\|_{0}$  &  $O(h^r)$ 
   &   $|\Pi^\nabla_h u-u_h|_{1}$ & $O(h^r)$ 
  \\ \hline 
    &  \multicolumn{4}{c}{ By the 3D $P_1$ stabilizer-free virtual element. } \\
\hline   
 5&  0.4944E-02&1.92&  0.2780E-01&1.95 \\
 6&  0.1254E-02&1.98&  0.7011E-02&1.99 \\
 7&  0.3145E-03&1.99&  0.1757E-02&2.00 \\
\hline
    &  \multicolumn{4}{c}{ By the 3D $P_2$ stabilizer-free virtual element. } \\
\hline   
 5&  0.1132E-04&3.89&  0.1001E-02&2.85 \\
 6&  0.7294E-06&3.96&  0.1313E-03&2.93 \\
 7&  0.4615E-07&3.98&  0.1680E-04&2.97 \\
\hline
    &  \multicolumn{4}{c}{ By the 3D $P_3$ stabilizer-free virtual element. } \\
\hline   
 4&  0.4149E-04&3.84&  0.1569E-02&2.83 \\
 5&  0.2688E-05&3.95&  0.2053E-03&2.93 \\
 6&  0.1703E-06&3.98&  0.2618E-04&2.97 \\
\hline
    &  \multicolumn{4}{c}{ By the 3D $P_4$ stabilizer-free virtual element. } \\
\hline   
 4&  0.7316E-06&4.94&  0.4987E-04&3.94 \\
 5&  0.2331E-07&4.97&  0.3169E-05&3.98 \\
 6&  0.7356E-09&4.99&  0.1997E-06&3.99 \\
\hline
    &  \multicolumn{4}{c}{ By the 3D $P_5$ stabilizer-free virtual element. } \\
\hline   
 3&  0.2986E-05&5.97&  0.7301E-04&4.96 \\
 4&  0.4707E-07&5.99&  0.2318E-05&4.98 \\
 5&  0.7386E-09&5.99&  0.7304E-07&4.99 \\
\hline 
    \end{tabular}%
\end{table}%

\section{Ethical Statement}

\subsection{Compliance with Ethical Standards} { \ }

   The submitted work is original and is not published elsewhere in any form or language.

\subsection{Funding } { \ }

Yanping Lin is supported in part by HKSAR GRF 15302922  and polyu-CAS joint Lab.

Mo Mu is supported in part by Hong Kong RGC CERG HKUST16301218. 

\subsection{Conflict of Interest} { \ }

  There is no potential conflict of interest .

\subsection{Ethical approval} { \ }

  This article does not contain any studies involving animals.
This article does not contain any studies involving human participants.
  
\subsection{Informed consent}  { \ }

This research does not have any human participant.  

\subsection{Availability of supporting data } { \ }

This research does not use any external or author-collected data.

\subsection{Authors' contributions } { \ }

All authors made equal contribution.
  
\subsection{Acknowledgments } { \ }

  None.

\end{document}